\documentclass[reqno, 12pt]{amsart}

\usepackage{amssymb}
\usepackage[usenames, dvipsnames]{color}

\input cyracc.def
\font \tencyr = wncyr10
\def\cyr{\tencyr\cyracc}

\newtheorem{theorem}{Theorem}[section]
\newtheorem{corollary}[theorem]{Corollary}
\newtheorem{lemma}[theorem]{Lemma}

\theoremstyle{definition}
\newtheorem{remark}[theorem]{Remark}

\theoremstyle{definition}

\theoremstyle{definition}
\newtheorem{assumption}[theorem]{Assumption}

\newcommand{\mysection}[1]{\section{#1}
\setcounter{equation}{0}}

\newcommand{\osc}{\mathop{\hbox{osc}}}

 \makeatletter
 \def\dashint{%
 \operatorname%
 {\,\,\text{\bf--}\kern-.98em\DOTSI\intop\ilimits@\!\!}}
 \makeatother

\newcommand\tr{\hbox{\rm tr}_{2}\,}

\def\sfa{{\sf a}}

\def\bM{\mathbb{M}}
\def\bQ{\mathbb{Q}}
\def\bR{\mathbb{R}}

\title[Elliptic and parabolic equations]
{Second-order elliptic and parabolic equations with
$B(\bR^{2}, VMO)$   coefficients }

\author{Hongjie Dong}
\address{Division of Applied Mathematics, Brown University,
182 George Street, Providence, RI 02912}
\email{Hongjie\_Dong@brown.edu}
\thanks{The work of the first author was partially supported by
NSF Grant DMS-0635607 from IAS and NSF Grant DMS-0800129. }

\author{N.V. Krylov}
\thanks{The work of the second author was partially supported by
NSF Grant DMS-0653121}

\address{127 Vincent Hall, University of Minnesota, Minneapolis, MN
 55455}
\email{krylov@math.umn.edu}
 \keywords{
Second-order elliptic and parabolic equations, vanishing mean oscillation,
  VMO coefficients, Sobolev spaces}

\subjclass[2000]{35K10, 35K20, 35J15}

\begin{document}

\begin{abstract}
The solvability in Sobolev spaces $W^{1,2}_p$ is proved for nondivergence
form second order parabolic equations for $p>2$ close to 2. The leading
coefficients are assumed to be measurable in the time variable and two
coordinates of space variables, and almost VMO (vanishing mean oscillation)
with respect to the other coordinates. This implies the $W^{2}_p$-solvability
for the same $p$ of nondivergence form elliptic equations with leading
coefficients measurable in two coordinates and VMO in the others. Under
slightly different assumptions, we also obtain the solvability results when
$p=2$.
\end{abstract}

\maketitle

\mysection{Introduction}
                                                \label{sec1}

In this paper, we consider the $W^{1,2}_{p}$-solvability of parabolic equations in nondivergence form:
\begin{equation}
                                                \label{parabolic}
L u-\lambda u=f,
\end{equation}
where $\lambda\ge 0$ is a constant, $f\in L_p$, and
\begin{equation}
                                                    \label{9.30.3}
L u=-u_t+a^{jk}D_{jk}u+b^{j}D_ju+cu.
\end{equation}
We assume that
 all the coefficients are bounded and measurable, and $a^{jk}$ are symmetric
and uniformly elliptic, i.e.
$$
|b^j|+|c|\le  K,
\quad a^{jk}=a^{kj},
\quad\delta|\xi|^{2}\leq
a^{jk}\xi^{j}\xi^{k}\leq\delta^{-1}|\xi|^{2}.
$$
If all the coefficients are time-independent, we also consider
the
$W^{2}_{p}$-solvability of elliptic equations in nondivergence
form:
\begin{equation}
                                                \label{elliptic}
M u-\lambda u=f,
\end{equation}
where
$$ M u=a^{jk}D_{jk}u+b^{j}D_ju+cu.
$$

We concentrate on rather irregular coefficients.
The Sobolev space theory of second-order parabolic and elliptic
equations with discontinuous coefficients was studied
extensively in the last thirty years. One important class of
discontinuous coefficients contains functions with vanishing
mean oscillation (VMO), the study of which was initiated in
\cite{CFL1} and continued in \cite{CFL2} and \cite{BC93}
(see also the references in \cite{Krylov_2007_mixed_VMO}).
The proofs in these references are based on the Calder\'on-Zygmund
theorem and the Coifman-Rochberg-Weiss commutator theorem.
Before that the Sobolev space theory had been established for
some other types of discontinuous coefficients; see, for
instance, \cite{Lo72a, Lo72b, Chi}.

In \cite{Krylov_2005}, the second author gave a unified
approach to investigating the $L_p$ solvability of both
divergence and nondivergence form parabolic and elliptic equations with
leading coefficients that are in VMO in the spatial
variables (and measurable in the time variable in the parabolic case).
Unlike the arguments in \cite{CFL1,CFL2,BC93}, the proofs in
\cite{Krylov_2005} rely mainly on pointwise estimates of sharp
functions of spatial derivatives of solutions. By doing this,
VMO coefficients are treated in a rather straightforward
manner. This method was later improved and generalized in a
series of papers
\cite{Krylov_2007_mixed_VMO,KK2,KK1,Ki1,Ki2,Ki3,Krylov08,DK08}.

The theory of elliptic and parabolic equations with partially
VMO coefficients was originated in \cite{KK2}. In \cite{KK2}
the $W^2_p$-solvability for any $p> 2$ was established for
nondivergence form elliptic equations with leading coefficients
measurable in one variable and VMO in the others. This result
was extended in \cite{KK1} to parabolic equations with leading
coefficients measurable in a spatial variable and VMO in the
others. For nondivergence form parabolic equations, more
general solvability results were obtained later in
\cite{Ki1,Ki2,Ki3}, in which most leading coefficients are
measurable in the time variable as well as one spatial
variable, and VMO in the other variables.

A natural question to ask is whether we still have the
$W^2_p$-solvability for elliptic equations if the leading
coefficients are measurable in {\em two} spatial variables and,
say, VMO in the others. Unfortunately, the answer is negative
for general $p>2$. Indeed, an example by Ural'tseva (see
\cite{LU}) tells us that even with  leading coefficients
depending only on the first two coordinates, there is no unique
solvability in $W^2_p$ for any fixed $p>2$ if the ellipticity
constant is sufficiently small. Nevertheless, Ural'tseva's
example does not rule out
 the possibility of  $W^2_p$-solvability for
$p$ sufficiently close to $2$ depending on the ellipticity
constant. This is the main motivation of our article.

In this article, we establish the solvability in Sobolev spaces
$W^{1,2}_p$ for nondivergence form second-order parabolic
equations for $p>2$ close to 2 (Theorem \ref{thm1}). The
leading coefficients are assumed to be measurable in the time
variable $t$ and two space variables $(x^1,x^2)$, and VMO with
respect to the others. Additionally, we assume $a^{11}+a^{22}$
is uniformly continuous with respect to $(x^1,x^2)$. This
result implies, in particular, the $W^{2}_p$-solvability for
the same $p$ of nondivergence form elliptic equations with
leading coefficients measurable in two coordinates and VMO in
the others (Theorem \ref{thm3}). Thus we give a positive answer
to the aforementioned question for those $p$ in a restricted
range. An interesting application of Theorem \ref{thm3}, shown
at the end of Section \ref{mainsec}, is the $W^2_p$-solvability
of elliptic equations in domains with rough coefficients. We
also investigate the case when $p=2$, which is of independent
interest. For elliptic equations, the $W^{2}_2$-solvability is
established when the leading coefficients are measurable
function of $(x^1,x^2)$ only. This extends a previously known
result proved in \cite{Chi} and \cite{KK2} where the
coefficients only depend  on $x^1$. For parabolic
equations, we obtain
the $W^{1,2}_2$-solvability under the
condition that $a^{jk}$ depend only on $(t,x^1,x^2)$ and
$a^{11}+a^{22}$ is uniformly continuous in $(x^1,x^2)$.

Next we give a brief description of our arguments. The proofs
are based on the aforementioned method from \cite{Krylov_2005}.
However, since $a^{ij}$ are merely measurable in $(x^1,x^2)$,
we are only able to estimate the sharp function of a portion of
the Hessian matrix $D^2u$ (Theorem \ref{thm4.1}), more
specifically, $D^2_{x''}u$ (see the beginning of the next
section for the notation). To bound the $W^{1,2}_p$ norm of the
solution by the $L_p$ norms of the right-hand side of the
equation and $D^2_{x''}u$, we use a result in \cite{Kr70}
proved for 2D parabolic equations with measurable coefficients.
These together with the $W^{1,2}_2$-solvability obtained in
Section \ref{sec3} enable us to establish the $W^{1,2}_p$
estimate.

An outline of the paper: in the next section, we introduce the
notation and state the main results, Theorem \ref{thm1},
\ref{thm3}, \ref{thm2},  and \ref{thm4}. Section \ref{sec2}
contains a few preliminary estimates. In Section \ref{sec3} we
establish the $W^{1,2}_2$-solvability and estimate the sharp
function of $D^2_{x''}u$. We finish the proof of
$W^{1,2}_p$-solvability in the last section by combining the
results in the previous sections.

\mysection{Main results}
                                \label{mainsec}

First we introduce some notation. Let $d\geq2$. A typical point
in $\bR^{d}$ is denoted by $x=(x^{1},...,x^{d})$. If $d\geq3$
we write
$x=(x',x'')$, where $x'=(x^{1},x^{2})$ and
$x''=(x^{3},...,x^{d})$.

We set
$$
D_{j}u=u_{x^j},\quad D_{jk}u=u_{x^jx^k},\quad
D_t u=u_t.
$$
By $Du$ and $D^{2}u$ we
mean the gradient and the Hessian matrix
of $u$. On many occasions we need to take these objects relative to only
part of variables. The reader understands
the meaning of the following notation which we use if  $d\geq3$:
$$
D_{x'}u=u_{x'},\quad D_{x''}u=u_{x''},\quad
D^{2}_{x' }u =D_{x'x'}u=u_{x'x'},
$$
$$
D_{x'x''}u=u_{x'x''},\quad
D^{2}_{ x'' }u=D_{x''x''}u=u_{x''x''}.
$$
For $-\infty\leq S<T\leq \infty$, we set
$$
W_{p}^{1,2}((S,T)\times \bR^d)=
\{u:\,u,u_t,Du,D^2u\in L_{p}((S,T)\times \bR^d)\},
$$
$$
W_{p}^{2}(\bR^d)=
\{u:\,u,Du,D^2u\in L_{p}(\bR^d)\},
$$
$$
\bR_T=(-\infty,T), \quad \bR_T^{d+1}=\bR_T\times \bR^d.
$$
We also use the abbreviations
$$
C^{\infty}_{0}=C^{\infty}_{0}(\bR^{d+1}),\quad
L_{p}=L_{p}(\bR^{d+1}),\quad
W^{1,2}_p=W^{1,2}_p(\bR^{d+1}),...
$$
For real- or complex- or matrix-valued functions $A(t,x)$ on $\bR^{d+1}$ we
understand $\|A\|_{L_{p}}^{p}$ as
$$
\int_{\bR^{d+1}}|\mbox{trace}\,A\bar{A}^{*}|^{p/2}\,dx\,dt.
$$
Accordingly are introduced the norms in $W$ spaces.

Our first two results concern the
 $W^{1,2}_2$- and  $W^{ 2}_2$-solvability of equations
\eqref{parabolic} and \eqref{elliptic}
 with measurable leading coefficients independent of $x''$.
 It seems to the authors that even these results
are new  if $d\geq3$.
Set
$$
\tr a=a^{11}+a^{22}.
$$
\begin{theorem}
                                         \label{thm1}
  Let $T\in (-\infty,+\infty]$. Assume
 that $a^{jk}$ depend only on $(t,x')$
and there exists an increasing function $\omega(r)$,
$r\geq0$,
such that $\omega(0+)=0$ and
$$
|\tr a(t,x')-\tr a(t,y')|\leq\omega(|x'-y'|)
$$
for all $t,x',y'$.
Then

i) There are constants $N=N(d,\delta,K,\omega)$ and
$\lambda_0=\lambda_0(d,\delta,K,\omega)\geq0$ such that
for any $u\in W^{1,2}_2(\bR^{d+1}_T)$ and $\lambda\geq \lambda_0$
we have
$$
\lambda\|u\|_{L_{2}(\bR^{d+1}_T)}+\sqrt{\lambda}
\|Du\|_{L_{2}(\bR^{d+1}_T)}+\|D^{2}u\|_{L_{2}(\bR^{d+1}_T)}
+\|u_t\|_{L_{2}(\bR^{d+1}_T)}
$$
\begin{equation}
                                                \label{9.22.1}
\leq N\|Lu-\lambda u\|_{L_{2}(\bR^{d+1}_T)}.
\end{equation}

ii) For any $\lambda> \lambda_0$ and $f\in L_2(\bR^{d+1}_T)$,
there exists a unique
solution $u\in W^{1,2}_2(\bR^{d+1}_T)$ of
 equation \eqref{parabolic} in $\bR^{d+1}_T$.

iii) In case $b^j\equiv c\equiv 0$
and $\tr a$ depends only on $t$,
 we can take $\lambda_0=0$ in
i) and ii).
\end{theorem}

Here is a similar result for elliptic equations.
\begin{theorem}
                                         \label{thm3}
Assume $a^{jk}=a^{jk}(x')$. Then

i) There are constants $N=N(d,\delta,K)$ and
$\lambda_0=\lambda_0(d,\delta,K)\geq0$ such that
for any $u\in W^{2}_2(\bR^{d})$ and $\lambda\geq \lambda_0$
we have
$$
\lambda\|u\|_{L_{2}(\bR^{d})}+
\sqrt{\lambda}
\|Du\|_{L_{2}(\bR^{d})}+
\|D^{2}u\|_{L_{2}(\bR^{d})}
\leq N\|Mu-\lambda u\|_{L_{2}(\bR^{d})}.
$$

ii) For any $\lambda> \lambda_0$ and $f\in L_2(\bR^{d})$, there exists a unique
solution $u\in W^{2}_2(\bR^{d})$ of   equation \eqref{elliptic}
in $\bR^{d}$.

iii) In case $b^j\equiv c\equiv 0$, we can take $\lambda_0=0$ in i) and
ii).
\end{theorem}

Theorem \ref{thm1} is proved in Section \ref{sec3}
and Theorem \ref{thm3}
is derived from it below in the present section.

\begin{remark}
						\label{rm1.3}
Theorem \ref{thm3} generalizes Theorem 2.5 of \cite{KK2}  and the main
result of \cite{Chi}, where the coefficients are independent of
$(x^{2},...,x^{d})$. From Theorem \ref{thm1} one can get Theorem 3.2
of \cite{KK1} where again the coefficients are independent of
$(x^{2},...,x^{d})$ but there is no restriction on
$\tr a$.   To show this we
introduce
  a new coordinate $y\in\bR$,   define
$$
\tilde L=L+(2\delta^{-1}-a^{11})D_{y}^2,
$$ and let $u(t,x,y)=u(t,x)\eta(y)$, where $\eta\in C_0^\infty(-2,2)$ is a nonnegative function and $\eta\equiv 1$ on $[-1,1]$. It is clear that
$$
\tilde L u(t,x,y)-\lambda u(t,x,y)=\tilde f,
$$
where
$$
\tilde f(t,x,y)=(Lu-\lambda u)(t,x)\eta(y)+(2\delta^{-1}-a^{11})u(t,x)\eta''(y).
$$
We now apply Theorem \ref{thm3} i) with $\tilde L$ and $u(t,x,y)$ in
place of $L$ and $u(t,x)$. With a sufficiently large $\lambda$, we will
arrive at \eqref{9.22.1} for function
$u(t,x)$.
 The remaining assertions of Theorem \ref{thm1}
in case the coefficients are independent of
$(x^{2},...,x^{d})$ are obtained as in its proof given
in Section \ref{sec3}. This argument also obviously applies
if $d=1$.
\end{remark}

\begin{remark}
                                    \label{rm10.17}
The conditions on $a^{jk}$ in Assertion i) and ii) of Theorem
\ref{thm3} can be relaxed. By using a partition of
unity and the method of freezing the coefficients, we can
 allow $a^{jk}$ to be measurable in $x'$ and uniformly
continuous in $x''$. In this case, the constants $\lambda_0$
and $N$ also depend on the modulus of continuity of $a^{jk}$
with respect to $x''$. Similarly, in Theorem \ref{thm1} we can
allow $a^{jk}$ to be measurable in
$(t,x')$, uniformly continuous in $x''$ and
$\tr a$ to be measurable in
$t$ and uniformly continuous in $x$.
\end{remark}

To state two more results we need some new notation.
If $B$ is a Borel subset of a hyperplane $\Gamma$ in a Euclidean space,
we denote by $|B|$ its volume relative to $\Gamma$. This notation is somewhat
ambiguous because $B$ also belongs to the ambient space,
where its volume can be zero. However, we hope that
from the context it will be clear relative to which hyperplane
we take the volume in each instance. If there is a measurable
function $f$ on $B$ which is integrable with respect to
Lebesgue measure $\ell$ on $\Gamma$ we set
$$
(f)_{B} =\dashint_{B}f( x)\,\ell(dx):=\frac{1}{|B|}
\int_{B}f( x)\,\ell(dx).
$$

 If $d\ge 3$, let
$$
B_r'(x') = \{ y' \in \bR^{2}: |x' -y' | < r\},
$$
$$
B_r''(x'') = \{ y'' \in \bR^{d-2}: |x''-y''| < r\},\quad B_{r}(x)=
B_r'(x')\times B_r''(x''),
$$
$$
Q_r(t,x) = (t-r^2,t) \times B_r(x),
$$
and let $\bQ$ be the collection of all $Q_r(t,x)$.
We call $r$ the radius of $Q=Q_r(t,x)$.
Set  $B_r'' = B_r''(0)$,
$B_r = B_r(0)$, $Q_r=Q_r(0,0)$.  If $d=2$, we denote $B_r(x)$ and
$Q_r(t,x)=(t-r^2,t) \times B_r(x)$ to be the usual balls and parabolic
cylinders. For a function $g$ defined on $\bR^{d+1}$, we denote its
(parabolic) maximal and sharp function, respectively, by
\begin{align*}
\bM g (t,x) &= \sup_{Q\in \bQ: (t,x) \in Q}
\dashint_{Q} | g(s,y) | \, dy \, ds,\\
g^{\#}(t,x) &= \sup_{Q\in \bQ: (t,x) \in Q}
\dashint_{Q} | g(s,y) - (g)_Q | \, dy \, ds.
\end{align*}

In the next theorem we require a quite mild regularity
assumption on $a^{jk}$. They are assumed to be measurable in
$t$ and $x'$, and almost VMO with respect to $x''$. More
precisely, we impose the following assumption
in which $\gamma>0$ will be specified later and
$R_{0}>0$ is a fixed number.

\begin{assumption}[$\gamma$]
                                            \label{assump2}
For any
$t,x,y$
 satisfying  $x''=y''$ and $|x'-y'|\leq R_{0}$ we have
\begin{equation}
                                                   \label{9.30.1}
 |\tr a(t,x)-
\tr a(t,y) | \le \gamma.
\end{equation}
Additionally if $d\geq3$,  for any
$Q=(s,t)\times B'\times B''\in\bQ$ with radius
$\rho\le R_0$
\begin{equation*}
\max_{j,k}\dashint_Q|a^{jk}(r,x)-\bar a^{jk}(r,x')|\,dx\,dr
\le \gamma
\end{equation*}
 where
$$
 \bar a^{jk}(r,x')=
\dashint_{B''} a^{jk}(r,x)\,dx''.
$$
\end{assumption}

\begin{theorem}
                                         \label{thm2}

One can find a
$\theta_0=\theta_0(\delta)>0$ such that for any
$p\in (2,2+\theta_0)$ there exists a $\gamma =\gamma(d,\delta,p) >0$
such that under Assumption
\ref{assump2} ($\gamma$) for any $T\in (-\infty,+\infty]$
 the following
holds.

i) For any $u\in W^{1,2}_p(\bR^{d+1}_T)$,
$$
\lambda\|u\|_{L_{p}(\bR^{d+1}_T)}+\sqrt{\lambda}
\|Du\|_{L_{p}(\bR^{d+1}_T)}+\|D^{2}u\|_{L_{p}(\bR^{d+1}_T)}
+\|u_t\|_{L_{p}(\bR^{d+1}_T)}
$$
\begin{equation}
                                                \label{13.10.21}
\leq N\|Lu-\lambda u\|_{L_{p}(\bR^{d+1}_T)},
\end{equation}
provided that $\lambda\geq \lambda_0$, where $\lambda_0
\geq0$ and $N$ depend only
 on $d,\delta,p,K$, and $R_0$.

ii) For any $\lambda> \lambda_0$ and $f\in L_p(\bR^{d+1}_T)$, there exists a unique solution
$u\in W^{1,2}_p(\bR^{d+1}_T)$ of equation \eqref{parabolic} in
$\bR^{d+1}_T$.

iii) In the case that
$a^{jk}=a^{jk}(t,x')$ and $b^j\equiv c\equiv 0$
 and $\tr a$   depends only on $t$,
we can
take
$\lambda_0=0$ in i) and ii).
\end{theorem}

 Theorem \ref{thm2} implies the solvability of the Cauchy
problem as in \cite{Krylov_2005}.  We prove Theorem \ref{thm2}
in Section \ref{sec4} and now we state one more result  for
elliptic equations in nondivergence form.

\begin{assumption}[$\gamma$]
                                            \label{assump3}
Either $d=2$ or $d\geq3$ and for any balls
$B'\subset\bR^{2},B''\subset\bR^{d-2}$
of the same radius $r\le R_0$,
\begin{equation*}
\sup_{j,k}\dashint_B|a^{jk}(x)-\bar a^{jk}(x')|\,dx\le \gamma,
\end{equation*}
where $B=B'\times B''$ and
$$
 \bar a^{jk}(x')=\dashint_{B''} a^{jk}(x )\,dx''.
$$
\end{assumption}

\begin{theorem}
                                         \label{thm4}
 Let $\theta_0$ be the constant in Theorem \ref{thm2}. Then for any $p\in
(2,2+\theta_0)$, there exists a $\gamma =\gamma(d,\delta,p) >0$ such that under
Assumption \ref{assump3} ($\gamma$) the following holds.

i) For any
$u\in W^{2}_p(\bR^{d})$,
$$
\lambda\|u\|_{L_{p}(\bR^{d})}+\sqrt{\lambda}
\|Du\|_{L_{p}(\bR^{d})}
+\|D^{2}u\|_{L_{p}(\bR^{d})}
\leq N\|Mu-\lambda u\|_{L_{p}(\bR^{d})},
$$
provided that $\lambda\geq \lambda_0$, where $\lambda_0 \geq0$ and $N$
depend only on $d,\delta,p,K$, and $R_0$
($R_{0}$ is excluded if $d=2$).

ii) For any $\lambda> \lambda_0$ and $f\in L_p(\bR^{d})$, there exists a unique
solution $u\in W^{2}_p(\bR^{d})$ of   equation \eqref{elliptic}
in $\bR^{d}$.

iii) In the case  that  $a^{jk}=a^{jk}(x')$ and $b^j\equiv c\equiv 0$, we
can take
$\lambda_0=0$ in i) and ii).
\end{theorem}

\begin{proof}[Proof of Theorem \ref{thm3} and \ref{thm4}] First
we assume that $\tr a$ is a constant. In this case, Theorem
\ref{thm3} and \ref{thm4} follow from Theorem \ref{thm1} and
\ref{thm2} respectively by using the idea that solutions to
elliptic equations can be viewed as steady state solutions to
parabolic equations. We omit the details and refer the reader
to the proof of Theorem 2.6
\cite{Krylov_2005}.

We now concentrate on proving Theorem \ref{thm4}
in the general case.  Owing to
 mollifications, a density
argument, and the method of continuity  it suffices to prove assertion i)
assuming that
 the coefficients are smooth and $u\in
C_0^\infty(\bR^{d})$
so that $f\in C_0^\infty (\bR^{d})$.  Note that the standard
mollification preserves Assumption \ref{assump3} ($\gamma$), the
ellipticity constant $\delta$, and the bounds $K$.  We introduce
$$
\tilde a^{jk}=\frac {a^{jk}} {\tr a},\quad \tilde b^{j}=\frac
{b^{j}} {\tr a},\quad \tilde c=\frac c {\tr a},
$$
and let $\tilde M$ be the elliptic operator constructed from
them. It is easy to see that $\tr \tilde a\equiv 1$, the new
coefficients satisfy the same boundedness and ellipticity
conditions with possibly different ellipticity constant and
bounds:
$\tilde \delta$ and $\tilde K$. Moreover, if $a^{jk}$ satisfy
Assumption \ref{assump3} ($\gamma$), then $\tilde a^{jk}$
satisfy Assumption
\ref{assump3} ($N(\delta)\gamma$).
Therefore, one can find a $\gamma>0$.
depending only on $d,\delta,p$, such that Assumption \ref{assump3}
($\gamma$) implies that for $\tilde
a^{jk}$
Assumption \ref{assump3} ($\gamma$)  is satisfied with
$\gamma=\gamma(d,\tilde{\delta},p)$ taken from Theorem \ref{thm2}.
Let $\tilde\lambda_0$ be the
constant from Theorem \ref{thm2}
corresponding to $\tilde \delta$ and $\tilde K$. Clearly
\eqref{elliptic} is equivalent to
$$
\tilde M u-{\lambda}u/ {\tr a} =f/ {\tr a}.
$$
For any $\lambda>2\delta^{-1}\tilde \lambda_0$, by the first
part of  the
proof there exists a unique $v\in W^2_p$ solving
$$
\tilde M v-\delta \lambda v/2=-| f/ {\tr a}|.
$$
Moreover, $v$ is a bounded classical solution since the
coefficients of $\tilde M$ are smooth and $| f /\tr a|$
is Lipschitz continuous. Due to the maximum principle $v\ge 0$
and $|u|\le v$ in $\bR^d$. Again, by the first part of the
proof, for appropriate $p$ and $N$ we have
\begin{equation}
 			\label{eq4.15}
\lambda \|u\|_{L_p (\bR^{d})}
\le \lambda \|v\|_{L_p (\bR^{d})}
\le N \|f\|_{L_p (\bR^{d})}.
\end{equation}
Since
$$
\tilde M u-\delta \lambda u/2= f/ {\tr a}+(1/\tr
a-\delta/2)\lambda u,
$$
we then obtain the desired estimate from the first
part of the proof and \eqref{eq4.15}.
This proves Theorem \ref{thm4}.

To prove Theorem \ref{thm3} it suffices to repeat the
above argument taking $p=2$ and dropping mentioning
Assumption \ref{assump3} ($\gamma$).
\end{proof}

An application of Theorem \ref{thm4} is the $W^2_p$-solvability of the Dirichlet
problem for the equation
$$
a^{jk}(x^1)D_{jk} u=f
$$
in $\{|x|<1\}$. Here
we assume that
$a^{jk}(x^1)$ are measurable in $x^1$ and continuous near $-1$ and $1$. This
equation can be solved by following the steps in Chapter 11 of
\cite{Kr08}. Notice that when locally flattening the boundary and using
odd/even extensions, one gets an equation with leading coefficient
measurable in two coordinates and continuous in the others.

\begin{remark}
                                                \label{remark 9.23.1}
The author of \cite{Ki1}-\cite{Ki3} presents quite general results
on the solvability of parabolic equations in Sobolev spaces
with or without mixed norms. Roughly speaking,
the main case in \cite{Ki1}-\cite{Ki3}
is when $a^{11}$ is measurable in   $x^{1}$ (or $t$) and VMO in
$(t,x^{2},...,x^{d})$ (or $x$)  and
$p$ is any number in $(2,\infty)$ without any restriction
on $\tr a$.
Theorem \ref{thm2} and the discussion in Remark  \ref{rm1.3} show that,
restricted to Sobolev spaces without mixed norms, some of D.~Kim's
results admit generalizations allowing $a^{11}$ which are measurable in
$(t,x^{1})$ and VMO in $(x^{2},...,x^{d})$ provided that
$p>2$ is close enough to 2. We have no idea what happens in this
situation if $p>2$ is arbitrary even if $d=1$.

In the case
of Theorem \ref{thm4} an example by Ural'tseva (see \cite{LU}) shows that
for any $d\geq 2$ its assertion becomes false for any fixed $p>2$ if
$\delta$ is sufficiently small.

\end{remark}

\mysection{Preliminary results}
                                            \label{sec2}
We first consider equations in $\bR\times \bR^2$ with measurable
coefficients.
\begin{lemma}
                                \label{lem2.1}
Let $T\in (-\infty,\infty]$, $d=2$ and
$$
Lu=-u_t+\sum_{j,k=1}^2 a^{jk}(t,x)D_{jk}u,
$$
  where $\tr a$ depends only on $t$. Then there exists
 a  $\theta_0=\theta_0(\delta)>0$ such that for
any
$p\in (2-\theta_0,2+\theta_0)$, $u\in W^{1,2}_p(\bR^3_T)$, and
$\lambda\ge 0$, we have
$$
\|D^{2}u\|_{L_{p}(\bR^3_T)}+\|u_t\|_{L_{p}(\bR^3_T)}
+\sqrt{\lambda}
\|Du\|_{L_{p}(\bR^3_T)}
$$
\begin{equation}
                                            \label{2.52}
+\lambda\|u\|_{L_{p}(\bR^3_T)}
\leq N\|Lu-\lambda u\|_{L_{p}(\bR^3_T)},
\end{equation}
where $N=N( \delta,p)$.
Moreover for $\lambda>0$ and $f\in L_p(\bR^3_T)$ there exists a unique $u\in W^{1,2}_p(\bR^3_T)$ solving $Lu-\lambda u=f$  in $\bR^3_T$.
\end{lemma}
\begin{proof}
First we consider the case that
 $T=\infty$. The change of variable
$$
t\to \frac 1 2\int_0^t  (\tr a)(s)\,ds
$$
 together with the argument in the proof of Theorem \ref{thm3}
and \ref{thm4} reduce  the problem to the case
when
$a^{11}+a^{22}=2$. Moreover, by a density argument to prove
\eqref{2.52} it suffices to consider
$u\in C_0^\infty$.
In case $u\in C_0^\infty(\Gamma)$  with
$\Gamma=(0,1)\times\{|x|<1\}$, it follows from Theorem 3 of
\cite{Kr70} that
$$
\lambda\|u\|_{L_{p}} +
\|u_t\|_{L_{p}}+\|D^{2}u\|_{L_{p}}
\leq N\|Lu -\lambda u\|_{L_{p}}.
$$
(See also \cite{Ca} for a result for elliptic equations.)
For general $u\in C_0^\infty$, we can use
shifting and scaling, the fact that the above $N$ depends only
on $\delta$ and $p$, and interpolation inequalities to treat
$Du$.
This proves \eqref{2.52}
if $T=\infty$. Adding using the standard
method of continuity completes the proof of
  the lemma when $T=\infty$.

For general $T$, we use the fact that $u=w$ for $t<T$, where $w\in
W_p^{1,2}$ solves
$$
L w-\lambda w=\chi_{t<T}(L u-\lambda u).
$$ The lemma is proved.
\end{proof}

An immediate corollary of Lemma \ref{lem2.1} is the following estimate.

\begin{corollary}
                                            \label{corollary 2.2}
Let $T\in (-\infty,\infty]$, $d\ge 3$,
$$
Lu=-u_t+\sum_{j,k=1}^d a^{jk}( t, x)D_{jk}u,
$$
where $\tr a $ depends only on $(t,x'')$. Then for any
$p\in (2-\theta_0,2+\theta_0)$, where $\theta_0$ is taken from
 Lemma \ref{lem2.1}, and any $u\in W^{1,2}_p(\bR^{d+1}_T)$ and $\lambda\ge
0$, we have
$$
\lambda\|u\|_{L_{p}(\bR^{d+1}_T)}+\sqrt{\lambda}
\|Du\|_{L_{p}(\bR^{d+1}_T)}+\|D^{2}u\|_{L_{p}(\bR^{d+1}_T)}
+\|u_t\|_{L_{p}(\bR^{d+1}_T)}
$$
\begin{equation}
                                            \label{2.57}
\leq N\|Lu-\lambda u\|_{L_{p}(\bR^{d+1}_T)}
+N\|D^{2}_{x''}u\|_{L_{p}(\bR^{d+1}_T)},
\end{equation}
where $N=N(\delta,d,p)$.
\end{corollary}
\begin{proof}
We first fix $x''$ and apply Lemma \ref{lem2.1} to get
$$
\lambda\|u(\cdot,\cdot,x'')\|^p_{L_{p}(\bR^3_T)}
+\|D_{x'}^{2}u(\cdot,\cdot,x'')\|^p_{L_{p}(\bR^3_T)}
+\|u_t(\cdot,\cdot,x'')\|^p_{L_{p}(\bR^3_T)}
$$
\begin{equation}
                                            \label{3.50}
\leq N\|  \sum_{j,k=1}^2 a^{jk}   D_{jk}u(\cdot,\cdot,x'')
-u_t (\cdot,\cdot,x'') -\lambda
u(\cdot,\cdot,x'')\|^p_{L_{p}(\bR^3_T)}.
\end{equation}
Upon integrating \eqref{3.50} with respect to $x''$ we arrive at
$$
\lambda\|u\|_{L_{p}(\bR^{d+1}_T)}+\|D_{x'}^{2}u\|_{L_{p}(\bR^{d+1}_T)}
+\|u_t\|_{L_{p}(\bR^{d+1}_T)}
$$
\begin{equation}
                                            \label{3.55}
\leq N\|Lu-\lambda u\|_{L_{p}(\bR^{d+1}_T)}
+\|D_{xx''}u\|_{L_{p}(\bR^{d+1}_T)}.
\end{equation}

 Observe that for any $\varepsilon>0$
\begin{equation}
                                                    \label{3.57}
\|D_{x'x''}u\|_{L_{p}}\le \varepsilon\|D_{x'x'}u\|_{L_{p}}
+N(d,p)\varepsilon^{-1}\|D_{x''x''}u\|_{L_{p}},
\end{equation}
 which is deduced from
$$
\|D_{x'x''}u\|_{L_{p}}\le N
\|\Delta u\|_{L_{p}}\leq N\|D_{x'x'}u\|_{L_{p}}+N \|D_{x''x''}u\|_{L_{p}}
$$
by scaling in $x'$. By using \eqref{3.57}, we  get from
\eqref{3.55}
$$
\lambda\|u\|_{L_{p}(\bR^{d+1}_T)}
+\|D_{x'x'}u\|_{L_{p}(\bR^{d+1}_T)}
+\|u_t\|_{L_{p}(\bR^{d+1}_T)}
$$
\begin{equation*}
\leq N\|Lu-\lambda u\|_{L_{p}(\bR^{d+1}_T)}
+\|D_{x''x''}u\|_{L_{p}(\bR^{d+1}_T)}.
\end{equation*}
To estimate $\|D_{x'x''}u\|_{L_{p}(\bR^{d+1}_T)}$ and $\|Du\|_{L_{p}(\bR^{d+1}_T)}$, we use \eqref{3.57} again and the interpolation inequality
$$
\sqrt \lambda \|Du\|_{L_{p}}\le N \lambda \|u\|_{L_{p}}+N\|D^2u\|_{L_{p}}.
$$
The corollary is proved.
\end{proof}

In the following theorem as in Corollary
\ref{corollary 2.2} the constant $\theta_{0}$ is taken
from  Lemma \ref{lem2.1}.

\begin{theorem}
                                              \label{theorem 9.30.1}
In case $d\geq3$ and $T\in(-\infty,\infty]$
for any $p\in (2-\theta_0,2+\theta_0)$ there exists a
$\gamma(d,p,\delta)>0$ such that, if
for any $t,x,y$,
satisfying $x''=y''$ and $|x'-y'|\leq R_{0}$,
condition \eqref{9.30.1} holds, then
estimate \eqref{2.57} is valid for any
$u\in W^{1,2}_p(\bR^{d+1}_T)$ and $\lambda\geq\lambda_{0}$
with general $L$ as in \eqref{9.30.3}
and $N$ and $\lambda_{0}\geq0$ depending only on $d,p,\delta$,
and $R_{0}$.  Furthermore, if $u(t,x)=0$
for $|x|\geq R_{0}$
  and $b^j=c=0$, then we can take $\lambda_0=0$ and $N$
to be independent of $R_0$.
\end{theorem}
\begin{proof} The idea is to use Corollary \ref{corollary 2.2}
in combination with
a standard method based  on freezing the coefficients and
 partitions of unity. We will show only the first step.
Assume that $u$ is of class $W^{1,2}_p $
and $u(t,x)=0$ for $|x|\geq R_{0}$.
Define $\sfa(t,x)=\sfa(t,x'')=a(t,0,x'')$ and
$$
L_{0}v=\tr\sfa(\tr a)^{-1}a^{jk}D_{jk}v-v_{t}.
$$
Then \eqref{2.57} holds with $L_{0}$ in place of $L$.
However, on the support of $u$
$$
|\tr\sfa(\tr a)^{-1}-1|\leq N(\delta)|\tr\sfa-\tr a|\leq
N(\delta)\gamma,
$$
so that
$$
\|L_{0}u-(a^{jk}D_{jk}u-u_{t})\|_{L_{p} }
\leq N(\delta)\gamma\|D^{2}u\|_{L_{p} },
$$
which shows how to choose $\gamma>0$ in order for this error
term times the $N$ from \eqref{2.57}
to be absorbed into the left-hand side of \eqref{2.57}.
\end{proof}

\mysection{Equations with coefficients measurable in $(t,x')$
and proof of Theorem \ref{thm1}}
                                            \label{sec3}

In this section we consider the operator
\begin{equation}
                                      \label{10.12.1}
Lu(t,x)=-u_t(t,x)+a^{jk}(t,x')D_{jk}u(t,x)
\end{equation}
assuming that $\tr a$ depends only on $t$.

First we generalize Theorem 2.5 of \cite{KK2}  (see also \cite{Chi})
proved for elliptic equations  with $a^{jk}$
depending only on one coordinate of $x$.

\begin{theorem}
                                         \label{theorem 9.8.1}
There is a constant $N=N(  \delta)$ such that
for any $u\in C^{\infty}_{0}$ and $\lambda\geq 0$
we have
$$
\lambda\|u\|_{L_{2}}+\sqrt{\lambda}
\|Du\|_{L_{2}}+\|D^{2}u\|_{L_{2}}+\|u_t\|_{L_{2}}
\leq N\|Lu-\lambda u\|_{L_{2}}.
$$
\end{theorem}

The case that $d=2$ is taken care of by
Lemma \ref{lem2.1}.
To prove the theorem in case $d\geq3$, we need some preparations.
To start with, we assume without loss of generality that
the coefficients $a^{jk}$ are infinitely differentiable
and have bounded derivatives.

 Set $f=Lu-\lambda u$ and let $\tilde{g}(t,x',\xi'')$
be the Fourier transform of a function $g(t,x)$
with respect to $x''$. Then
$$
-\tilde
u_t(t,x',\xi'')+\sum_{j,k=1}^{2}a^{jk}(t,x')D_{jk}\tilde{u}(t,x',\xi'')
+i\sum_{j =1}^{2}B^{j}(t,x',\xi'')D_{j}\tilde{u}(t,x',\xi'')
$$
\begin{equation}
                                      \label{9.8.3}
-C(t,x',\xi'') \tilde{u}(t,x',\xi'')=\tilde{f}(t,x',\xi''),
\end{equation}
where
$$
B^{j}(t,x',\xi'')=2\sum_{k >2}a^{jk}(t,x')\xi^{k} ,
\quad C(t,x',\xi'')=\lambda+\sum_{j,k>2}
a^{jk}(t,x')\xi^{j}\xi^{k}.
$$

In the following lemma $\xi$ is considered
 as a parameter.

\begin{lemma}
                                         \label{lemma 9.8.1}
Let $d\geq3$ and $|\xi''|^2 + \lambda > 0$.
Then we have
\begin{equation}
                                                  \label{3.10.3}
|\tilde{u}(t,x',\xi'')|\leq\hat{u}(t,x',\xi''),
\end{equation}
where, for each $\xi''\in\bR^{d-2}$, $\hat{u}(t,x',\xi'')$
is the unique bounded classical
solution of
$$
-\hat{u}_t(t,x',\xi'')+\sum_{j,k=1}^{2}a^{jk}(t,x')D_{jk}\hat{u}(t,x',\xi'')
$$
 \begin{equation}
                                      \label{9.8.1}
 -( \lambda
 + \delta |\xi''|^{2} )\hat{u}(t,x',\xi'') =
-|\tilde{f}(t,x',\xi'')|  .
\end{equation}

Furthermore,
 \begin{equation}                             \label{9.8.2}
( |\xi''|^{2}
+\lambda) \|\tilde{u}( \cdot, \cdot,\xi'')\| _{L_{2}(\bR\times\bR^{2})}
\leq N(\delta)\|\tilde{f}( \cdot, \cdot,\xi'')\| _{L_{2}(\bR\times\bR^{2})}.
\end{equation}
\end{lemma}

\begin{proof} The idea of the proof is to
eliminate the first-order  terms in \eqref{9.8.3}
by using probability theory and Girsanov's transformation.
Let $a'$ be the $2\times 2$ matrix, which stands at
 the upper left corner of $a$.
Set $\sigma=\sqrt{2a'}$.
Fix a point  $(t_0,x')$  and let $x'_{t}$ be the solution
of the following It\^o's equation
$$
x'_{t}=x'+\int_{0}^{t}\sigma(t_0-s,x'_{s})\,dw_{s}
$$
on a probability space carrying a two-dimensional
Wiener process $w_{t}$. Also set
$$
B=(B^{1},B^{2}),\quad \hat{B}= B\sigma^{-1} ,
$$
$$
\rho_{t}(\xi'')=\exp\big(i\int_{0}^{t}\hat{B}(t_0-s,x'_{s},\xi'')\,dw_{s}
$$
$$
-\int_{0}^{t}(C(t_0-s,x'_{s},\xi'')-(1/2)|\hat{B}(t_0-s,x'_{s},\xi'')|^{2})
\,ds\big).
$$
As is easy to check by using It\^o's formula
and \eqref{9.8.3}
$$
d\big(\rho_{t}(\xi'')\tilde{u}(t_0-t,x'_{t},\xi'')\big)
=\rho_{t}(\xi'')\tilde{f}(t_0-t,x'_{t},\xi'')\, dt
$$
$$
+[i\rho_{t}(\xi'')\tilde{u}(t_0-t,x'_{t},\xi'')\hat{B}(t_0-t,x'_{t},\xi'')
$$
$$
+\rho_{t}(\xi'')
D\tilde{u}(t_0-t,x'_{t},\xi'')\sigma(t_0-t,x'_{t})]\,dw_{t}.
$$
We integrate this relation between $0$ and $T\in(0,\infty)$
and take expectations of the result. Then, since $\tilde{u}$
and $\hat{B}$ are bounded ($\xi''$  is fixed),
the expectation of the stochastic integral disappears and we
obtain
\begin{equation}
                                                  \label{9.8.5}
\tilde{u}(t_0,x' ,\xi'')=
E\rho_{T}(\xi'')\tilde{u}(t_0-T,x'_{T},\xi'')-E\int_{0}^{T}
\rho_{t}(\xi'')\tilde{f}(t_0-t,x'_{t},\xi'')\, dt.
\end{equation}
Next observe that
$$
 \delta|\xi''|^{2}\leq
\delta|\xi  |^{2}\leq a^{jk}(t,x')\xi^{j}\xi^{k}
$$
$$
=(1/2)|\sigma(t,x') \xi'|^{2}+\sum_{j=1}^{2}\xi^{j}B^{j}
(t,x',\xi'')
+C(t,x',\xi'')-\lambda.
$$
Substituting $ \xi'\to\sigma^{-1}(t,x')\xi' $ we see that
$$
0\leq(1/2) |\xi'|^{2}
+\sum_{j=1}^{2}\xi^{j}\hat{B}^{j}
(t,x',\xi'')
+C(t,x',\xi'')-\lambda-\delta|\xi''|^{2}.
$$
Since this is true for any $\xi'$, we have
$$
\big|(1/2)\sum_{j=1}^{2}\xi^{j}\hat{B}^{j}
(t,x',\xi'')\big|^{2}\leq(1/2) |\xi'|^{2}
(C(t,x',\xi'')-\lambda-\delta|\xi''|^{2}),
$$
$$
(1/2) | \hat{B}
(t,x',\xi'') |^{2}\leq
 C(t,x',\xi'')-\lambda-\delta|\xi''|^{2}  ,
$$
implying that
$$
|\rho_{t}(\xi'')|\leq e^{-(\lambda+\delta|\xi''|^{2})t}.
$$
Therefore, passing to the limit as $T\to\infty$ in \eqref{9.8.5}
we obtain
$$
\tilde{u}(t_0,x' ,\xi'')=
 -E\int_{0}^{\infty}
\rho_{t}(\xi'')\tilde{f}(t_0-t,x'_{t},\xi'')\, dt,
$$
$$
|\tilde{u}(t_{0},x' ,\xi'')|\leq
 E\int_{0}^{\infty}
 |\tilde{f}(t_0-t,x'_{t},\xi'')|e^{-(\lambda+\delta|\xi''|^{2})t}\, dt
=:\hat{u}(t_0,x',\xi'') .
$$

 Next notice that
 equation \eqref{9.8.1}
  has a unique solution in $W^{1,2}_{2}(\bR\times \bR^{2})$ by Lemma
\ref{lem2.1}. It is a bounded classical  solution
 since $u\in C^{\infty}_{0}$, $a$ is smooth, and $\tilde{f}$ is
Lipschitz continuous in
$(t,x')$. This solution is the above
$\hat{u}$ which is proved by using It\^o's formula in the same way as above.
Estimate \eqref{9.8.2} for $\hat{u}$ in place of $\tilde{u}$
follows from Lemma \ref{lem2.1}. Having \eqref{9.8.2} for $\hat{u}$
in place of $\tilde{u}$ gives us \eqref{9.8.2} as is.
The lemma is proved.
\end{proof}

\begin{remark}
Inequality \eqref{3.10.3} can also be proved without using
probability theory along the following lines. By the maximum principle,
we have $\hat u\ge 0$. Fix $\xi''$ and let
$$
\Omega=\{(t,x')\in \bR\times \bR^2\,:\,{\tilde u}(t,x',
\xi'')\neq 0\},
$$
which is open and bounded. For any $(t,x')\in \Omega$, $|{\tilde u}|$
 has  continuous first derivatives in $(t,x')$ and
  second derivatives in $x'$
 and we have
$$
D_t|{\tilde u}|=\frac 1 {2|{\tilde u}|}( \bar{{\tilde u}}D_t{\tilde u}+{\tilde u}
D_t\bar{{\tilde u}}),
$$
$$
D_j|{\tilde u}|=\frac 1 {2|{\tilde u}|}(\bar{{\tilde u}}D_j{\tilde u} +{\tilde u}
D_j\bar{{\tilde u}}),
$$
$$
D_{jk}|{\tilde u}|=\frac 1 {2|{\tilde u}|}(\bar{{\tilde
u}}D_{jk}{\tilde u} +{\tilde u}D_{jk}\bar{{\tilde u}}
+(D_j{\tilde u}) D_k\bar{{\tilde u}}+(D_k{\tilde u})
D_j\bar{{\tilde u}})
$$
$$
-\frac 1 {4|{\tilde u}|^3}(\bar{{\tilde u}}D_j{\tilde u}+{\tilde u} D_j\bar{{\tilde
u}})(\bar{{\tilde u}}D_k{\tilde u}+{\tilde u} D_k\bar{{\tilde u}} ).
$$
Therefore, by \eqref{9.8.3}
$$
-D_t|\tilde u|+\sum_{j,k=1}^2 a^{jk}D_{jk}|{\tilde u}|=\text{Re}(-D_t|\tilde u|+a^{jk}D_{jk}|{\tilde u}|)
$$
$$=\frac 1 {|{\tilde u}|} \text{Re}(\bar {\tilde u} \tilde f+C|{\tilde u}|^2)-\frac 1 {|{\tilde u}|}\sum_{j=1}^2 B^j \text{Im}({\tilde u} D_j \bar {\tilde u})
$$
$$
+\sum_{j,k=1}^2 \frac {a^{jk}} {4|{\tilde u}|^3}[2|{\tilde u}|^2
((D_j{\tilde u}) D_k\bar{{\tilde u}}+(D_k{\tilde u}) D_j\bar{{\tilde u}})
$$
$$
-(\bar{{\tilde u}}D_j{\tilde
u}+{\tilde u} D_j\bar{{\tilde u}})(\bar{{\tilde u}}D_k{\tilde u}+{\tilde
u} D_k \bar{{\tilde u}})].
$$
The last sum on the right-hand side above is equal to
$$
\sum_{j,k=1}^2\frac {a^{jk}} {4|{\tilde u}|^3}[|{\tilde u}|^2
((D_j{\tilde u}) D_k\bar{{\tilde u}}+(D_k{\tilde u}) D_j\bar{{\tilde u}})-\bar{{\tilde
u}}^2 (D_j{\tilde u}) D_k{\tilde u}-
{\tilde u}^2 (D_j\bar {\tilde u}) D_k \bar {\tilde u}]
$$
$$
=\sum_{j,k=1}^2-\frac {a^{jk}} {4|{\tilde u}|^3}({\tilde u} D_j \bar {\tilde u}-\bar {\tilde u} D_j {\tilde u})({\tilde u} D_k \bar {\tilde u}-\bar {\tilde u} D_k {\tilde u})
$$
$$
=\sum_{j,k=1}^2\frac {a^{jk}} {|{\tilde u}|^3} \text{Im}({\tilde u} D_j \bar {\tilde u}) \text{Im}({\tilde u} D_k \bar {\tilde u}).
$$
Thus,
$$
-D_t|\tilde u|+\sum_{j,k=1}^2 a^{jk}D_{jk}|{\tilde u}|\ge -|\tilde f|+
\big(\lambda +\sum_{j,k>2}a^{jk}(x')\xi^j\xi^k\big)|{\tilde u}|
$$
$$
-\frac 2{|{\tilde u}|}\sum_{j=1}^2 \sum_{k>2}a^{jk}(x')\xi^k \text{Im}({\tilde u} D_j \bar {\tilde u})
+\sum_{j,k=1}^2\frac {a^{jk}} {|{\tilde u}|^3} \text{Im}({\tilde u} D_j \bar {\tilde u}) \text{Im}({\tilde u} D_k \bar {\tilde u})
$$
$$
\ge -|\tilde f|+\lambda |{\tilde u}|+\delta|\xi''|^2 |{\tilde u}|.
$$
In the last inequality we used the uniform ellipticity condition.
By the maximum principle, we obtain \eqref{3.10.3}.
\end{remark}

\begin{proof}[Proof of Theorem \ref{theorem 9.8.1}]
Recall that we may assume   $d\geq3$.
By squaring both  sides  of \eqref{9.8.2}, integrating with respect to
$\xi''$, and using Parseval's identity we obtain
$$
\lambda^{2}\|u \|_{L_{2}}^{2}+
\|u_{x''x''}\|_{L_{2}}^{2}\leq N(\delta)\|f\|_{L_{2}}^{2},
$$
which along with Corollary \ref{corollary 2.2}
 proves the theorem with a constant
$N$ perhaps depending on $d$ and $\delta$.

To show that it is independent of $d$, we
 use Lemma \ref{lem2.1} to get
$$
\|\tilde{u}_{x'x'}(\cdot,\cdot,\xi'')\|^{2}_{L_{2}(\bR\times\bR^{2})}
+\|\tilde{u}_t(\cdot,\cdot,\xi'')\|^{2}_{L_{2}(\bR\times\bR^{2})}
$$
$$
\leq N
\big\|-\tilde{u}_t+\sum_{j,k=1}^{2}a^{jk} D_{jk}\tilde{u}(\cdot,\cdot,\xi'')
\big\|^{2}_{L_{2}(\bR\times\bR^{2})}.
$$
We also use that
$$
|B(t,x',\xi'')|\leq N|\xi''|,
\quad C(t,x',\xi'')
\leq N(\lambda+|\xi''|^{2}).
$$
Then from Lemma \ref{lemma 9.8.1} and \eqref{9.8.3}
we conclude that for $\xi''\ne0$
$$
\|\tilde{u}_{x'x'}(\cdot,\cdot,\xi'')\|^{2}_{L_{2}(\bR\times\bR^{2})}
+\|\tilde{u}_t(\cdot,\cdot,\xi'')\|^{2}_{L_{2}(\bR\times\bR^{2})}
$$
$$
\leq N (|\xi''|^{2}\|\tilde{u}_{x' }(\cdot,\cdot,\xi'')
\|^{2}_{L_{2}(\bR\times\bR^{2})}+
\|\tilde{f} (\cdot,\cdot,\xi'')\|^{2}_{L_{2}(\bR\times\bR^{2})}).
$$
Here for any $\varepsilon>0$
$$
 |\xi''|^{2}\|\tilde{u}_{x' }(\cdot,\cdot,\xi'')
\|^{2}_{L_{2}(\bR\times\bR^{2})}
$$
$$
\leq
\varepsilon\|\tilde{u}_{x'x' }(\cdot,\cdot,\xi'')
\|^{2}_{L_{2}(\bR\times\bR^{2})}+
N\varepsilon^{-1}|\xi''|^{4}
\|\tilde{u} (\cdot,\cdot,\xi'')
\|^{2}_{L_{2}(\bR\times\bR^{2})}
$$
\begin{equation}
                                             \label{8.9.6}
\leq
\varepsilon\|\tilde{u}_{x'x' }(\cdot,\cdot,\xi'')
\|^{2}_{L_{2}(\bR\times\bR^{2})}+N
\|\tilde{f} (\cdot,\cdot,\xi'')\|^{2}_{L_{2}(\bR\times\bR^{2})}.
\end{equation}
It follows that
$$
\|\tilde{u}_{x'x'}(\cdot,\cdot,\xi'')\|^{2}_{L_{2}(\bR\times\bR^{2})}
+\|\tilde{u}_t(\cdot,\cdot,\xi'')\|^{2}_{L_{2}(\bR\times\bR^{2})}
\leq N
\|\tilde{f} (\cdot,\cdot,\xi'')\|^{2}_{L_{2}(\bR\times\bR^{2})} .
$$
Upon integrating this inequality with respect to $\xi''$
and using Parseval's identity we arrive at
$$
\|u_{x'x'}\|_{L_{2}}^{2}+\|u_{t}\|_{L_{2}}^{2}\leq N\|f\|_{L_{2}}^{2}.
$$
Going back to \eqref{8.9.6} and integrating again
we obtain
$$
\|u_{x'x''}\|_{L_{2}}^{2}\leq N\|f\|_{L_{2}}^{2}.
$$
After that, to finish proving the theorem, it
only remains to combine the above estimates
and use   the interpolation inequality
$$
\lambda^2 \|Du\|_{L^2}^2\le \lambda^4 \|u\|_{L^2}^2+\|\Delta u\|_{L^2}^2=
\lambda^4 \|u\|_{L^2}^2+\|D^2 u\|_{L^2}^2.
$$
The theorem is proved.
\end{proof}

Next we give a proof of Theorem \ref{thm1}.
\begin{proof}[Proof of Theorem \ref{thm1}]
Part iii) follows from the first two by using a scaling argument. By the
same reason as in the proof of Lemma \ref{lem2.1}, it suffices to prove
i) and ii) for $T=\infty$. In case
$T=\infty$,
assertion  i) is obtained from Theorem \ref{theorem 9.8.1}
 in the way outlined in the proof
of Theorem \ref{theorem 9.30.1} and assertion ii)
is obtained by the method of continuity.
 \end{proof}

Next, we go back
to considering the operator $L$ introduced in \eqref{10.12.1}
and present the key results
of this section.

\begin{theorem}
                                     \label{theorem 9.10.1}
Let $d\geq3$, $\kappa \ge 2$, and $r > 0$.
Assume that $u \in
C_0^{\infty}
$  and $Lu =0$ in $Q_{\kappa r}$.
Then there exist constants
$N = N(d, \delta)$ and $\alpha=\alpha(d,\delta)\in (0,1]$ such that
for any multi-index  $\gamma=(\gamma',\gamma'')$
\begin{equation*}
\dashint_{Q_r} |D^{\gamma''}u  -
 (D^{\gamma''}u)_{Q_r}|^2 \, dx\,dt
\le N \kappa^{-2\alpha} \left(|D^{\gamma''}u|^{2}\right)_{Q_{\kappa r}}.
\end{equation*}
\end{theorem}

\begin{proof} By observing that $LD^{\gamma''}u=0$ we see that
it suffices to concentrate on $\gamma=0$. By using scaling
we reduce the general situation to the one in which $r=1$.
By Lemma 4.2.4 of \cite{Kr85}
and Theorem 7.21 of \cite{Li}
$$
\osc _{Q_{1/\kappa}}u
\le N\kappa^{-\alpha} \|u \|_{L_2(Q_1)}
$$
with $\alpha$ and $N$ as in the statement. Scaling this estimate
shows that
$$
\osc _{Q_{1}}u
\le N \kappa^{-\alpha}\left(| u|^{2}\right)_{Q_{\kappa }}^{1/2}.
$$
It only remains to observe that
$$
\dashint_{Q_1} | u  -
 ( u)_{Q_1}|^2 \, dx\,dt\leq {\color{red}N}
 (\osc _{Q_{1}}u)^2 .
$$
The theorem is proved.
\end{proof}

\begin{theorem}
                                        \label{thm12.02}
Let $d\geq3$ and let
$\alpha$ be the constant in Theorem \ref{theorem 9.10.1}. Then there is
a constant $N$ depending only on $d,\delta$ such that for any $u\in
W^{1,2}_{2,{\rm loc}}$, $r\in (0,\infty)$, and $\kappa\ge 4$
$$
\left(|u_{x''x''}(t,x)-(u_{x''x''})_{Q_r}|^2\right)_{Q_r}
$$
\begin{equation}
                                \label{eq12.05}
\le N\kappa^{d+2}
\left(| Lu|^2\right)_{Q_{\kappa
r}}+N\kappa^{-2\alpha}
\left(|u_ {x''x'' }|^2\right)_{Q_{\kappa
r}}.
\end{equation}
\end{theorem}
\begin{proof}
Fix an $r\in (0,\infty)$ and a
$\kappa\ge 4$. We may certainly assume that $a^{jk}$ are infinitely differentiable and
have bounded derivatives. Also changing $u$ for large $|t| + |x|$ does not affect
\eqref{eq12.05}. Therefore, we may assume that $u \in {\color{red}W^{1,2}_p}{\color{blue}W^{1,2}_2}$
 and moreover $u \in C^\infty_0 $.

Now we define $f= Lu-\lambda u\in C_0^\infty $. Take a $\zeta\in
C_0^\infty $ such that $\zeta=1$ in $Q_{\kappa r/2}$ and $\zeta=0$
outside the closure of $Q_{\kappa r}\cup (-Q_{\kappa r})$. Define $v$ to
be the unique $W^{1,2}_2((S,T)\times\bR^{d})$-solution of the equation
$$
Lv=(1-\zeta)f
$$
with zero initial condition at $t=S$, where $S<-\kappa r$ and $ T>\kappa r$ are
such that $u(t,x)=0$ for
$t\not\in(S,T)$. By classical theory we know that such a $v$ indeed exists and
is unique and infinitely differentiable. Since $(1-\zeta)f= 0$ in $Q_{\kappa
r/2}$ and $\kappa/2\ge 2$, by Theorem \ref{theorem 9.10.1} with $v$ in place of
 $u$  we obtain
$$
\left(|v_{x''x''} -
 (v_{x''x''} )_{Q_r}|^2\right)_{Q_{r}}
$$
\begin{equation}
                                          \label{13.4.45}
\le N \kappa^{-2\alpha}
\left(|v_ {x''x'' }|^{2}\right)_{Q_{\kappa r/2}}
\le N \kappa^{-2\alpha}
\left(|v_ {x''x''} |^{2}\right)_{Q_{\kappa r}}.
\end{equation}

On the other hand, obviously $\bar{v}(t,x):=v(t,x)I_{(S,T)}(t)$
is of class $W^{1,2}_{2}(\bR^{d+1}_{T})$ and the function $w := u -\bar{v} \in W^{1,2}_2(\bR^{d+1}_{T})$ satisfies
$$
Lw=\zeta f
$$
in $\bR^{d+1}_{0}$.
Therefore, by Theorem \ref{thm1} (iii),
$$
\int_{\bR_0^{d+1}} |w_{xx}|^2 \,dx\,dt
\le N\int_{\bR_0^{d+1}}|\zeta f|^2\,dx\,dt
\le N\int_{Q_{\kappa r}}|f|^2\,dx\,dt,
$$
which implies
\begin{equation}
                                        \label{eq13.4.41}
\left(|w_{xx}|^{2}\right)_{Q_{\kappa r}}\le
N\left(|f|^{2}\right)_{Q_{\kappa r}},
\end{equation}
and
\begin{equation}
                                        \label{13.4.46}
\left(|w_{xx}|^{2}\right)_{Q_{r}}
\le N\kappa^{d+2}\left(|f|^{2}\right)_{Q_{\kappa r}}.
\end{equation}
Combining \eqref{13.4.45}-\eqref{13.4.46} together, we conclude
$$
\left(|u_{x''x''} -
 (u_{x''x''} )_{Q_r}|^2\right)_{Q_{r}}\le \left(|v_{x''x''} -
 (v_{x''x''}
)_{Q_r}|^2\right)_{Q_{r}}+
N\left(|w_ {x''x'' }|^{2}\right)_{ Q_{r}  }
$$
$$
\le N \kappa^{-2\alpha} \left(|v_ {x''x''} |^{2}\right)_{Q_{\kappa r}}
+N\kappa^{d+2}\left(|f|^{2}\right)_{Q_{\kappa r}}
$$
$$
\le N \kappa^{-2\alpha} \left(|u_ {x''x'' }|^{2}\right)_{Q_{\kappa r}}
+N\kappa^{d+2}\left(|f|^{2}\right)_{Q_{\kappa r}}.
$$
The theorem is proved.
\end{proof}

\mysection{Proof of Theorem \ref{thm2}}
                                            \label{sec4}
We shall prove Theorem \ref{thm2} in this section. We consider the
operator
$$
L u=-u_t+a^{jk}D_{jk}u+b^{j}D_ju+cu,
$$
where $a^{jk}$ satisfy Assumption \ref{assump2} ($\gamma$) with
some $\gamma>0$ to be specified later. Assertion (iii) of
Theorem \ref{thm2}     is  obtained
from (i) and (ii) by using scaling. Assertion (ii)
is obtained from
(i) by the method of continuity. If $d=2$, assertion (i)
is derived from Lemma \ref{lem2.1} in a standard way
alluded to a few times before. Therefore, it only
 remains to prove assertion (i)
assuming that $d\geq3$.

Set
$$
L_0 u=-u_t+a^{jk}D_{jk}u.
$$

First we generalize Theorem \ref{thm12.02}.
\begin{theorem}
                                            \label{thm4.1}
Let $\alpha$ be the constant in
Theorem \ref{theorem 9.10.1}, $\gamma > 0$, $\tau,\sigma \in (1,\infty)$,
$1/\tau+1/\sigma=1$. Take a $u\in W^{1,2}_2$ and set $f=L_0 u$. Then under
Assumption
\ref{assump2} ($\gamma$) there exists a positive constant $N$ depending only on
$d$,
$\delta$, and $\tau$ such that, for any
$(t_0,x_0)\in \bR^{d+1}$,
$r\in (0,\infty)$, and $\kappa\ge 4$,
$$
\left(|u_{x''x''}(t,x)-(u_{x''x''})_{Q_r(t_0,x_0)}|^2\right)_{Q_r(t_0,x_0)}
$$
$$
\le N\kappa^{d+2}
\left(|f|^2\right)_{Q_{\kappa r}(t_0,x_0)}
+N\kappa^{d+2}\gamma^{1/\sigma}
\left(|u_{xx}|^{2\tau}\right)_{Q_{\kappa r}(t_0,x_0)}^{1/\tau}
$$
\begin{equation}
                                \label{eq13.5.05}
+N\kappa^{-2\alpha}
\left(|u_ {x''x'' }|^2\right)_{Q_{\kappa r}(t_0,x_0)},
\end{equation}
provided that $u$ vanishes outside  $Q_{R_0}$.

\end{theorem}
\begin{proof}
We fix $(t_0,x_0)\in \bR^{d+1}$, $\kappa\ge 4$, and $r\in
(0,\infty)$.
  Choose $Q=(s_{1},s_{2})\times B'\times B''$ to be  $Q_{\kappa
r}(t_0,x_0)$ if $\kappa r< R_0 $ and   $Q_{ R_0 }$ if $\kappa
r\ge  R_0 $. Recall the definition of
$\bar{a}(t,x')$ given in Assumption \ref{assump2}, set $y'_{0}$
to be the center of $B'$, and introduce
$\bar{a}_{0}(t)=\bar{a}(t,y_{0}')$,
$$
\sfa^{jk}=\frac{\bar a^{jk}}{\tr\bar a} \tr \bar{a}_{0} ,\quad
\hat{f}= \sfa ^{jk}D_{jk}u-u_{t}.
$$
Obviously, $\sfa$ depends only on $(t,x')$,
$\tr\sfa=\tr \bar{a}_{0}$
 depends only on $t$ and takes values
  between $2\delta$ and $2\delta^{-1}$, and
$$
\dashint_{Q}|a^{jk}- \sfa ^{jk}|\,dx\,dt\leq N
\dashint_{Q}|a^{jk}\tr \bar{a}- \sfa ^{jk}\tr  \bar{a} |\,dx\,dt
$$
\begin{equation}
                                                   \label{9.16.1}
\leq N\dashint_{Q}|a^{jk}- \bar a ^{jk}|\,dx\,dt+
N\dashint_{Q}|\tr \bar{a}-  \tr \bar{a}_{0}
|\,dx\,dt\leq N\gamma.
\end{equation}
Also
$$
\hat{f} =  \left(\sfa^{jk} - a^{jk}\right) D_{jk}u + f.
$$

Then by Theorem \ref{thm12.02} with an appropriate translation  and $\sfa$ in
place of $a$,
$$
\dashint_{Q_r(t_0,x_0)} | u_{x''x''}
- \left( u_{x''x''} \right)_{Q_r(t_0x_0)} |^2 \, dx\,dt
$$
\begin{equation}							
                                \label{13.5.42}
\le N\kappa^{d+2}
\left(| \hat f|^2\right)_{Q_{\kappa r}(t_0,x_0)}
+N\kappa^{-2\alpha}
\left(|u_ {x''x'' }|^2\right)_{Q_{\kappa r}(t_0,x_0)},
\end{equation}
where $N$ and $\alpha$ depend only on $d$ and $\delta$.
By the definition of $\hat f$,
\begin{equation}							\label{13.5.52}
\int_{Q_{\kappa r}(t_0,x_0)} |\hat{f}|^2 \, dx\,dt
\le 2 \int_{Q_{\kappa r}(t_0,x_0)} |f|^2 \, dx\,dt
+2 I,
\end{equation}
where
$$
I =
\int_{Q_{\kappa r}(t_0,x_0)}
\big| (\sfa^{jk} - a^{jk}) D_{jk}u \big|^2 \, dx\,dt
$$
$$
= \int_{Q_{\kappa r}(t_0,x_0) \cap Q_{ R_0 }}
\big| (\sfa^{jk} - a^{jk}) D_{jk}u \big|^2 \, dx\,dt.
$$
By   H\"older's inequality, we have
\begin{equation}							\label{13.5.53}
I \le NI_1^{1/\sigma} I_2^{1/\tau},
\end{equation}
where
$$
I_1 = \sum_{j,k}
\int_{Q_{\kappa r}(t_0,x_0) \cap Q_{ R_0}}
| \sfa^{jk} - a^{jk} |^{2\sigma} \, dx\,dt,
\quad
I_2 = \int_{Q_{\kappa r}(t_0,x_0)} |u_{xx}|^{2\tau} \, dx\,dt.
$$
According to \eqref{9.16.1} we have
$$
I_1 \leq \sum_{j,k}\int_{Q }
| \sfa^{jk} - a^{jk} |^{2\sigma} \, dx\,dt\leq N\gamma|Q|
\le  N  (\kappa r)^{d+2} \gamma.
$$
This together with  \eqref{13.5.42}-\eqref{13.5.53} yields
\eqref{eq13.5.05}. The theorem is proved.
\end{proof}

\begin{remark}
                                         \label{remark 10.2.1}
Assume that $a^{jk}$ ($=a^{kj}$) are independent of $x''$
for $j=1,2$ and $k=1,...,d$. Also assume that
$\tr a$ depends only on $t$. Then in the above
proof we have
 $\sfa^{jk}=\bar{a}^{jk}$ for all $j,k$
and
$\sfa^{jk}=a^{jk}$ for $j=1,2$ and $k=1,...,d$.
Therefore, in the definition of $I$  we only need
to sum over $j,k\geq3$, so that only $u_{x''x''}$
are involved in its estimate. It follows that
in this case we can replace $|u_{xx}|^{2\tau}$
in \eqref{eq13.5.05} with $|u_{x''x''}|^{2\tau}$.
\end{remark}

\begin{lemma}
                                                 \label{lem4.2}
Let $\theta_0$ be the constant
in Lemma \ref{lem2.1},
$p\in (2,2+\theta_0)$ and $f\in L_p$. Then there exist
strictly positive constants
$\gamma$  and $N$  both
 depending only on $d,p$, and $\delta$ such that under
Assumption \ref{assump2} ($\gamma$), for any $u\in W^{1,2}_p$
vanishing outside $Q_{ R_0}$   and satisfying
$L_0 u =f$, we have
$$
\|u_t\|_{L_p}+\|D^2u\|_{L_p}\le N\|f\|_{L_p}.
$$

\end{lemma}
\begin{proof}
Let $\alpha$ be the constant in Theorem \ref{theorem 9.10.1}.
Choose $\tau \in (1, \infty)$ such that $p > 2\tau$.
  Inequality \eqref{eq13.5.05} implies that
on $\bR^{d+1}$
$$
u_{x''x''}^{\#}
\le N \kappa^{(d+2)/2} \bM^{1/2}(|f|^2)
$$
$$
+  N \kappa^{(d+2)/2}\gamma^{1/(2\sigma)} \bM^{1/(2\tau)}(|u_{xx}|^{2\tau})
+N\kappa^{-\alpha}\bM^{1/2}(|u_{xx}|^{2}).
$$
We apply the Fefferman-Stein theorem on sharp functions
and the Hardy-Littlewood maximal function theorem
to the above inequality to get
$$
\| u_{x''x''} \|_{L_p}
\le N \| u_{x''x''}^{\#} \|_{L_p}
\le N \kappa^{(d+2)/2} \| \bM(|f|^2) \|_{L_{p/2}}^{1/2}
$$
$$
+ N\kappa^{(d+2)/2}\gamma^{1/(2\sigma)} \|\bM(|u_{xx}|^{2\tau})
\|_{L_{p/(2\tau)}}^{1/(2\tau)} +N\kappa^{-\alpha}\| \bM(|u_{xx}|^2)
\|_{L_{p/2}}^{1/2}
$$
$$
\le N \kappa^{(d+2)/2} \| f \|_{L_p} + N \left(\kappa^{(d+2)/2}
\gamma^{1/(2\sigma)}+\kappa^{-\alpha}\right)
\|u_{xx}\|_{L_{p}},
$$
where in the last inequality we use the fact that $p > 2\tau$.
From this estimate  and the last assertion of
Theorem \ref{theorem 9.30.1}   we have
$$
\|u_{xx}\|_{L_p}+\|u_{t}\|_{L_p}
\le N \kappa^{(d+2)/2} \| f \|_{L_p}
$$
$$
 + N \left(\kappa^{(d+2)/2}
\gamma^{1/(2\sigma)}+\kappa^{-\alpha}\right)
\|u_{xx}\|_{L_{p}} .
$$
To finish the proof of the lemma, it suffices to choose a big enough $\kappa$ and then a small $\gamma$
so that
$$
N \left(\kappa^{(d+2)/2} \gamma^{1/(2\sigma)}+\kappa^{-\alpha}\right) \le
1/2.
$$\end{proof}

Now we are in the position to prove Theorem \ref{thm2}.

\begin{proof}[Proof of Theorem \ref{thm2}]
As is pointed out {\color{red}in} {\color{blue}at} the beginning of the section,
it suffices to prove assertion  i)  for $d\geq3$.
 Similarly to the proof of assertions i)
and ii) of Theorem \ref{thm1},  we only need  to prove
\eqref{13.10.21} for $T=\infty$ and  $u\in C_0^\infty$. This in turn is
obtained from Lemma
\ref{lem4.2} by using a partition of unity and
an idea of Agmon (see, for instance,  Section 6.3 of
\cite{Kr08}).
\end{proof}

We finish the paper with
a statement valid for any $p>2$ which
partially generalizes
Lemma \ref{lem4.2}.
Its proof is an
immediate consequence of Remark \ref{remark 10.2.1}
and the argument from the proof of Lemma \ref{lem4.2}.
\begin{theorem}
                                     \label{theorem 10.2.1}
Assume that $d\geq3$ and
 $a^{jk}$ ($=a^{kj}$) are independent of $x''$
for $j=1,2$ and $k=1,...,d$. Also assume that
$\tr a$ depends only on $t$. Let $p>2$ and
$u\in W^{1,2}_p$ be such that $u$ vanishes outside $Q_{ R_0}$.
Then there are a
strictly positive constant
$\gamma$  and $N$  both
 depending only on $d,p$, and $\delta$ such that under
Assumption \ref{assump2} ($\gamma$)    we have
$$
 \|D^2_{x''}u\|_{L_p}\le N\|L_{0}u\|_{L_p}.
$$
\end{theorem}


\end{document}